\documentclass[final,1p,times]{elsarticle}

\usepackage{amsthm,amsmath,amssymb,amsfonts,graphicx,graphics,latexsym,exscale,cmmib57,dsfont,bbm,amscd,ulem}
\usepackage{bbm}
\usepackage{mathrsfs}
\usepackage{color,soul}
\usepackage[colorlinks=true]{hyperref}

\newcommand{\bP}{\boldsymbol{P}}
\newcommand{\p}{\boldsymbol{p}}
\newcommand{\bdelta}{\boldsymbol{\delta}}
\newcommand{\A}{\boldsymbol{A}}

\newcommand{\bS}{\boldsymbol{S}}

\newcommand{\bxi}{\boldsymbol{\xi}}

\newcommand{\bea}{\begin{eqnarray}}
\newcommand{\eea}{\end{eqnarray}}
\newcommand{\bean}{\begin{eqnarray*}}
\newcommand{\eean}{\end{eqnarray*}}
\newtheorem{thm}{Theorem}[section]

\newtheorem*{que}{Question}

\newtheorem{lemma}[thm]{Lemma}
\newtheorem{cor}[thm]{Corollary}

\theoremstyle{definition}

\numberwithin{equation}{section}

\journal{xxx}

\begin{document}

\begin{frontmatter}

\title{Exponential stability of nonhomogeneous matrix-valued Markovian chains}

\author[D]{Xiongping Dai}
\ead{xpdai@nju.edu.cn}
\author[TH]{Tingwen Huang}
\ead{tingwen.huang@qatar.tamu.edu}
\author[YH]{Yu Huang}
\ead{stshyu@mail.sysu.edu.cn}
\address[D]{Department of Mathematics, Nanjing University, Nanjing 210093, People's Republic of China}
\address[TH]{Department of Mathematics, Texas A$\&$M University at Qatar, PO Box 23874, Doha, Qatar}
\address[YH]{Department of Mathematics, Zhongshan (Sun Yat-Sen) University, Guangzhou 510275, People's Republic of China}
\begin{abstract}
Let $\bxi=\{\xi_n\}_{n\ge0}$ be a nonhomogeneous/nonstationary Marvovian chain on a probability space $(\Omega,\mathscr{F},\mathbb{P})$ valued in the state space $\bS$ that consists of a finite number of real $d$-by-$d$ matrices such that $\mathbb{P}(\{\xi_0=S\})>0$ for each $S\in\bS$. As usual, $\bxi$ is called \textit{uniformly exponentially stable} if there exist two constants $C>0$ and $0<\lambda<1$ so that for all $n\ge1$, $\|\xi_0(\omega)\xi_1(\omega)\dotsm\xi_{n-1}(\omega)\|\le C\lambda^{n}$ for $\mathbb{P}\textrm{-a.e. }\omega\in\Omega$.
In this note, we show that if the Markovian transition probability matrices of $\bxi$ have the same transition sign matrix for all times $n\ge0$, then $\bxi$ is uniformly exponentially stable if and only if there are $\gamma<1$ and $N>0$ such that for each $n>N$, the spectral radii $\rho(S_{i_0}\dotsm S_{i_{n-1}})$ are less than or equal to $\gamma$ for all $n$-length closed sample paths $(S_{i_0},\dotsm, S_{i_{n-1}})\in\bS^n$ with $\mathbb{P}(\{\xi_0=S_{i_0},\dotsc,\xi_{n-1}=S_{i_{n-1}},\xi_n=S_{i_0}\})>0$.
\end{abstract}


\end{frontmatter}

\section{Introduction}\label{sec1}%
Let $(\Omega,\mathscr{F},\mathbb{P})$ be a probability space, and let $\bS=\{S_1,\dotsc,S_K\}$ be a finite subset of the real $d$-by-$d$ matrix space $\mathbb{R}^{d\times d}$ endowed with the discrete topology, where $2\le K, d<\infty$ both are integers. We consider a nonhomogeneous matrix-valued Markovian chain
\begin{equation*}
\bxi=\{\xi_n\}_{n\ge0},\quad \textrm{where }\xi_n\colon\Omega\rightarrow\bS\textrm{ are random variables}.
\end{equation*}
Here the nonhomogeneity means that the Markovian transition probability matrices $\{\bP(n)\}_{n\ge0}$ of $\bxi$ are time-varying. Write
\begin{equation*}
\bP(n)=\left(p_{ij}(n)\right)_{1\le i,j\le K}\in\mathbb{R}^{K\times K},
\end{equation*}
where
\begin{equation*}
p_{ij}(n)=P(\{\xi_{n+1}=S_j|\xi_n=S_i\})\quad\textrm{if }\mathbb{P}(\{\xi_n=S_i\})\not=0.
\end{equation*}
For $\mathbb{P}$-a.e. $\omega\in\Omega$, the stability problem of the infinite-length sample paths
$\bxi(\omega)=\{\xi_n(\omega)\}_{n\ge0}$ is an important issue in both pure and applied mathematics.
The purpose of this note is just to characterize when $\bxi$ is exponentially stable at $\mathbb{P}$-a.e. $\omega\in\Omega$.

\medskip
\textbf{Notations}: By $\rho(A)$ is meant the spectral radius of a square matrix $A$, which is defined as the maximum of the absolute values of all the eigenvalues of $A$.

For $n\ge1$, an $n$-length sample path $(S_{i_0},\dotsc,S_{i_{n-1}})\in\bS^n$ of $\bxi$ is said to be \textit{non-ignorable} provided that $\mathbb{P}(\{\xi_0=S_{i_0},\dotsc,\xi_{n-1}=S_{i_{n-1}}\})>0$.

For any finite-length sample path $(S_{i_0},\dotsc,S_{i_{n-1}})$ in $\bS^n$, it is called a \textit{non-ignorable closed sample path} of $\bxi$, if the $(n+1)$-length sample path $(S_{i_0},\dotsc,S_{i_{n-1}},S_{i_0})$ in $\bS^{n+1}$ is non-ignorable for $\bxi$, i.e. $\mathbb{P}(\{\xi_0=S_{i_0},\dotsc,\xi_{n-1}=S_{i_{n-1}},\xi_n=S_{i_0}\})>0$. For a non-ignorable closed sample path $(S_{i_0},\dotsc,S_{i_{n-1}})$, it may be extended into a non-ignorable \textit{periodic path} of $\bxi$ under additional condition like having constant transition sign matrix defined below:
$$(\overbrace{S_{i_0},\dotsc,S_{i_{n-1}}},\overbrace{S_{i_0},\dotsc,S_{i_{n-1}}},\overbrace{S_{i_0},\dotsc,S_{i_{n-1}}},\dotsc).$$
So a closed sample path of $\bxi$ is also called a periodic sample path of $\bxi$.

Let $\p^{(0)}=\left(p_1^{(0)},\dotsc,p_K^{(0)}\right)$ be the initial probability vector of $\bxi$, i.e. $p_k^{(0)}=\mathbb{P}(\{\xi_0=S_k\})$ for $1\le k\le K$. We say $\p^{(0)}$ is \textit{irreducible} if $p_k^{(0)}>0$ for all $1\le k\le K$.

Since on most occasions we will not matter the explicit value of $\bP(n)$, we now introduce an essential condition as follows. Let $\mathrm{sign}(x)$ stand for the sign function; i.e., $\mathrm{sign}(x)=1$ if $x>0$ and $=0$ if $x\le0$. We say that the Markovian chain $\bxi$ has the constant transition sign matrix if
the $K$-by-$K$ $\{0,1\}$-matrices
\begin{equation*}
\mathbb{S}=(s_{ij})\equiv\big{(}\mathrm{sign}(p_{ij}(n))\big{)}_{1\le i,j\le K}\quad \forall n\ge0,
\end{equation*}
are independent of the times $n$. That means that for any two states $S_i,S_j\in\bS$, if the transition probability $p_{ij}(0)$ of $\bxi$ from the state $S_i$ to state $S_j$ at time $n=0$ in one unit time is strictly positive, then the transition probability $p_{ij}(n)>0$ at all other times $n\ge1$.

Here $\mathbb{S}$ is called the \textit{transition sign matrix} of $\bxi$.

\subsection{Motivations}\label{sec1.1}
We note that for the dynamical behaviors of a Markovian chain $\bxi$, ones are only interested to $\mathbb{P}$-almost every sample points $\omega\in\Omega$, neither every nor an explicit sample point $\omega$ in $\Omega$.

For example, let $K=2, d=1$, $\p^{(0)}=(1/2,1/2)$ and
$$
\mathbb{S}=\left(\begin{matrix}0&1\\1&0\end{matrix}\right)\quad \textrm{and}\quad S_1=2,\ S_2=\frac{1}{3}.
$$
Then although $\bS=\{S_1,S_2\}\subset\mathbb{R}^{1\times1}$ is not stable itself because $S_1$ is of unstable mode, yet for $\mathbb{P}$-a.e. $\omega\in\Omega$ we have
$$(\xi_n(\omega))_{n\ge0}=\textrm{either }(2,\frac{1}{3},2,\frac{1}{3},2,\frac{1}{3},\dotsc)\,\textrm{ or }\,(\frac{1}{3},2,\frac{1}{3},2,\frac{1}{3},2,\dotsc);$$
hence $\bxi$ is uniformly exponentially stable for $\mathbb{P}$-a.e. $\omega\in\Omega$, but not for every $\omega\in\Omega$. In this case, the stability is completely determined by the two non-ignorable closed sample paths $(S_1,S_2)$ and $(S_2,S_1)$ with $\mathbb{P}(\{\xi_0=S_1,\xi_1=S_2\})=\frac{1}{2}$ and $\mathbb{P}(\{\xi_0=S_2,\xi_1=S_1\})=\frac{1}{2}$.

For any sample point $\omega\in\Omega$, if $\|\xi_0(\omega)\dotsm\xi_n(\omega)\|$ converges to $0$ as $n\to\infty$, then we say $\bxi$ is stable at the sample point $\omega$. Since this involves computing the limit of $\xi_0(\omega)\dotsm\xi_n(\omega)$ as $n\to\infty$, it is not an easy task to judge the stability of $\bxi$. However, if the infinite sample path $(\xi_n(\omega))_{n\ge0}$ is periodically generated by a finite-length sample path, say $(S_{j_0},\dotsc,S_{j_{m-1}})\in\bS^m$, i.e., $$(\xi_n(\omega))_{n\ge0}=(\overbrace{S_{j_0},\dotsc,S_{j_{m-1}}},\overbrace{S_{j_0},\dotsc,S_{j_{m-1}}},\dotsc),$$
then $\bxi$ is stable at $\omega$ if and only if $\rho(S_{j_0},\dotsc,S_{j_{m-1}})<1$.

The question arises immediately of \textit{whether one can judge the stability of $\bxi$ via the non-ignorable periodic elements of the Markovian chain $\bxi$, in general}.

This problem has been raised and studied by many peoples since 1980s; cf.~\cite{PR,Gur,LM,BTT,SWMWK} and so on. We shall positively answer this problem in this note.
\subsection{Main statements}\label{sec1.2}
In this note, we shall obtain the following sufficient and necessary condition for the uniform exponential stability of $\bxi$.

\begin{thm}\label{thm1}
Let $\bxi$ be a nonhomogeneous matrix-valued Markovian chain, which has the constant transition sign matrix and irreducible initial probability vector. Then the following three statements are equivalent to each other.
\begin{enumerate}
\item[$(1)$] $\bxi$ is uniformly exponentially stable.
\item[$(2)$] There are constants $\gamma<1$ and $N>0$ such that for each $n>N$,
\begin{equation*}
\rho(S_{i_0}\dotsm S_{i_{n-1}})\le\gamma
\end{equation*}
for all non-ignorable $n$-length sample paths $(S_{i_0},\dotsm, S_{i_{n-1}})\in\bS^n$.
\item[$(3)$] There are constants $\gamma<1$ and $N>0$ such that for each $n>N$,
\begin{equation*}
\rho(S_{i_0}\dotsm S_{i_{n-1}})\le\gamma
\end{equation*}
for all non-ignorable $n$-length closed sample paths $(S_{i_0},\dotsm, S_{i_{n-1}})\in\bS^n$.
\end{enumerate}
\end{thm}

This characterizes the uniform exponential stability.
Here the uniform exponential stability of $\bxi$ is defined in the same way as in the Abstract; that is to say, there exist two universal constants $C>0$ and $0<\lambda<1$ such that
\begin{equation*}
\|\xi_0(\omega)\dotsm\xi_{n-1}(\omega)\|\le C\lambda^{n}\quad\forall n\ge1\textrm{ and }\mathbb{P}\textrm{-a.e. }\omega\in\Omega.
\end{equation*}
This means that the sequence of product matrices $\xi_0(\omega)\dotsm\xi_n(\omega)$ converges exponentially fast to $0$ uniformly for $\mathbb{P}$-a.e. $\omega\in\Omega$.

Weakly, if there holds that the so-called Lyapunov exponents
\begin{equation*}
    \lambda_{\bxi}(\omega):=\lim_{n\to\infty} \frac{1}{n}\log\|\xi_0(\omega)\dotsm\xi_{n-1}(\omega)\|<0,\quad \mathbb{P}\textrm{-a.e. }\omega\in\Omega,
\end{equation*}
then $\bxi$ is referred to as \textit{exponentially stable $\mathbb{P}$-almost surely}. In this weak case, for $\mathbb{P}$-a.e. $\omega\in\Omega$, $\xi_0(\omega)\dotsm\xi_n(\omega)$ also converges exponentially fast to $0$, but not necessarily uniformly for $\mathbb{P}$-a.e. $\omega\in\Omega$.

The Markovian chain $\bxi$ is said to be \textit{irreducible}\footnote{This is different from the \textit{irreducibility of $\bS$} that says there exists no a nontrivial, common, invariant, proper subspace of $\mathbb{R}^d$ for each member $S_k$ of $\bS$.} if its transition probability matrices $\bP(n)$ are irreducible. This is equivalent to say that the transition sign matrix $\mathbb{S}$ of $\bxi$ is irreducible; i.e., for any pair $1\le i,j\le K$, there is some integer $N=N_{ij}>0$ such that the $(i,j)$-entry of the $N$-folds product $\mathbb{S}^N$ is strictly positive.

Then we shall obtain the following sufficient condition for the $\mathbb{P}$-almost sure exponential stability of $\bxi$.

\begin{thm}\label{thm2}
Let $\bxi$ be a nonhomogeneous matrix-valued Markovian chain, which has the irreducible constant transition sign matrix $\mathbb{S}$ and irreducible initial probability $\p^{(0)}$. Then $\bxi$ is exponentially stable $\mathbb{P}$-almost surely, if for any $n\ge1$,
\begin{equation*}
\rho(S_{i_0}\dotsm S_{i_{n-1}})<1
\end{equation*}
for all non-ignorable $n$-length closed sample paths $(S_{i_0},\dotsm, S_{i_{n-1}})\in\bS^n$.
\end{thm}

These two theorems show that the non-ignorable closed sample paths or periodic elements of $\bxi$ may completely decide the exponential stability of the nonstationary random process $\bxi$.

\subsection{Outline}\label{sec1.3}
Let $\bS^{\mathbb{Z}_+}=\left\{(S_{i_n})_{n=0}^{\infty}\,|\,1\le i_n\le K, S_{i_n}\in\bS\right\}$ be endowed with the standard compact product topology. Then
\begin{equation*}
\Theta_+\colon\bS^{\mathbb{Z}_+}\rightarrow\bS^{\mathbb{Z}_+}\quad \textrm{by }(S_{i_n})_{n=0}^{\infty}\mapsto(S_{i_{n+1}})_{n=0}^{\infty}
\end{equation*}
is a continuous surjective shift transformation. By the random variable
\begin{equation*}
\Pi_{\bxi}\colon\Omega\rightarrow\bS^{\mathbb{Z}_+};\quad \omega\mapsto(\xi_n(\omega))_{n=0}^\infty,
\end{equation*}
we can obtain the probability distribution $\mathbb{P}^{\Pi_{\bxi}}$ on the Borel measurable space $(\bS^{\mathbb{Z}_+},\mathscr{B}_{\bS^{\mathbb{Z}_+}})$ defined by
\begin{equation*}
\mathbb{P}^{\Pi_{\bxi}}([S_{j_0},\dotsc,S_{j_{m-1}}])=\mathbb{P}(\{\xi_0=S_{j_0},\dotsc,\xi_{m-1}=S_{j_{m-1}}\})
\end{equation*}
for any $m$-length blocks $[S_{j_0},\dotsc,S_{j_{m-1}}]=\{(S_{i_n})_{n=0}^\infty\,|\,S_{i_0}=S_{j_0},\dotsc,S_{i_{m-1}}=S_{j_{m-1}}\}$, for any $m\ge1$. Then there follows the claim:
\begin{itemize}
\item The Markovian chain $\bxi$ is uniformly exponentially stable if and only if there are constants $C>0$ and $\gamma<1$ such that
\begin{equation*}
\|S_{i_0}\dotsm S_{i_{m-1}}\|\le C\lambda^{m},\quad\forall m\ge1,\ \mathbb{P}^{\Pi_{\bxi}}\textrm{-a.e. }(S_{i_n})_{n=0}^\infty\in\bS^{\mathbb{Z}_+}.
\end{equation*}
\end{itemize}
If $\bxi$ is homogeneous and stationary, then $\mathbb{P}^{\Pi_{\bxi}}$ is invariant left by $\Theta_+$, i.e. $\mathbb{P}^{\Pi_{\bxi}}=\mathbb{P}^{\Pi_{\bxi}}\circ\Theta_+^{-1}$; in other words,
\begin{equation*}
\mathbb{P}^{\Pi_{\bxi}}([S_{j_0},\dotsc,S_{j_{m-1}}])=\sum_{1\le i_0\le K}\mathbb{P}^{\Pi_{\bxi}}([S_{i_0},S_{j_0},\dotsc,S_{j_{m-1}}]).
\end{equation*}
In our present situation, however, $\mathbb{P}^{\Pi_{\bxi}}$ is not $\Theta_+$-invariant. This in turn suggests that the classical ergodic theory and the Oselede\v{c} multiplicative ergodic theorem cannot play directly a role in proving the above theorems.

Two prominent tools in the analysis of stability of matrix random products without any constraints are the so-called Barabanov norm~\cite{Bar} and Elsner reduction theorem~\cite{Els}. However, in our situation, there are no analogs of Barabanov's norm and Elsner's reduction theorem.

To get around these points mentioned above, we shall construct a stationary canonical Markovian probability measure which is equivalent to $\mathbb{P}^{\Pi_{\bxi}}$. Particularly, to prove the Gel'fand-Berger-Wang formula of a topological Markovian chain \cite{Dai-pre}, Kozyakin has recently introduced a new approach---$\{0,1\}$-matrix lift of a matrix-valued topological Markovian chain in \cite{Koz-pre}. His idea makes us to overcome the obstructions caused by lacking Barabanov's norm and Elsner's reduction theorem.

The rest of this note is organized as follows. In Section~\ref{sec2}, we shall equivalently formulate our statements in terms of of matrix-valued topological Markovian chain. This is very convenient for us to employ some known criteria of stability established for linear switched dynamical systems. In Section~\ref{sec3}, we shall introduce Kozyakin's $\{0,1\}$-matrix lift of a matrix-valued topological Markovian chain. It builds up for us a bridge between a matrix-valued topological Markovian chain and the case of completely free product of matrices. In Section~\ref{sec4}, we will complete the proofs of our main theorems using Kozyakin's idea combining with some stability criteria for periodically stable linear switched systems. Finally we will end this note with concluding remarks in Section~\ref{sec5}.

\section{Matrix-valued topological Markovian chain}\label{sec2}
To prove our Theorems~\ref{thm1} and \ref{thm2}, we need to describe them in terms of matrix-valued topological Markovian chains.
Throughout this section, let $\bxi$ be a nonhomogeneous matrix-valued Markovian chain defined on $(\Omega,\mathscr{F},\mathbb{P})$ valued in the state space $\bS=\{S_1,\dotsc,S_K\}$ as in the Theorems~\ref{thm1} and \ref{thm2} stated in Section~\ref{sec1}.

Let $\mathbb{S}=(s_{k\ell})_{1\le k,\ell\le K}$ be the $K\times K$ constant transition sign matrix of $\bxi$. By the definition of $\mathbb{S}$, there is at least one entry $1$ at each row of $\mathbb{S}$, and it gives rise to a subshift of finite type as follows:

Let $\varSigma_{\mathbb{S}}^+=\left\{(i_n)_{n=0}^\infty\colon i_n\in\{1,\dotsc,K\}\textrm{ and }s_{i_ni_{n+1}}=1\textrm{ for all }n\ge0\right\}$,
which is nonempty and compact as a subspace of the compact product topological space $\varSigma_K^+=\{1,\dotsc,K\}^{\mathbb{Z}_+}$. Then there is the natural Markovian shift transformation of finite type
\begin{equation*}
\theta_+\colon\varSigma_{\mathbb{S}}^+\rightarrow\varSigma_{\mathbb{S}}^+;\quad (i_n)_{n=0}^\infty\mapsto(i_{n+1})_{n=0}^\infty.
\end{equation*}
For any $n$-length word $(i_0,\dotsc,i_{n-1})\in\{1,\dotsc,K\}^n$ where $n\ge2$, it is called \textit{$\mathbb{S}$-admissible}, if $s_{i_ki_{k+1}}=1$ for all $0\le k<n-1$.

Then there holds the following basic result:

\begin{lemma}\label{lem2.1}
Given any $n$-length word $(i_0,\dotsc,i_{n-1})\in\{1,\dotsc,K\}^n$ where $n\ge2$, it is $\mathbb{S}$-admissible if and only if the event
$\left\{\xi_0=S_{i_0},\dotsc,\bxi_{n-1}=S_{i_{n-1}}\right\}$
is non-ignorable for $\bxi$.
\end{lemma}

\begin{proof}
We first note that for any $n$-length word $(i_0,\dotsc,i_{n-1})\in\{1,\dotsc,K\}^n$ where $n\ge2$,
\begin{equation*}
\mathbb{P}(\{\xi_0=S_{i_0},\dotsc,\xi_{n-1}=S_{i_{n-1}}\})=p_{i_0}^{(0)}p_{i_0i_1}(0)\dotsm p_{i_{n-2}i_{n-1}}(n-2),
\end{equation*}
where $\p^{(0)}=\left(p_1^{(0)},\dotsc,p_K^{(0)}\right)$ is the initial probability vector of $\bxi$ as in Section~\ref{sec1}. Since $\p^{(0)}$ is irreducible, $(i_0,\dotsc,i_{n-1})$ is $\mathbb{S}$-admissible if and only if $p_{i_0i_1}(0)\dotsm p_{i_{n-2}i_{n-1}}(n-2)>0$.

This completes the proof of Lemma~\ref{lem2.1}.
\end{proof}

Recall that $(i_0,\dotsc,i_{n-1})\in\{1,\dotsc,K\}^n$ for $n\ge1$ is said to be \textit{$\mathbb{S}$-periodically extendable} if it is $\mathbb{S}$-admissible and in addition $s_{i_{n-1}i_0}=1$. This means that the periodic sequence of period $n$ $(i_0,\dotsc,i_{n-1},i_0,\dotsc,i_{n-1},\dotsc)$ belongs to $\varSigma_{\mathbb{S}}^+$.

By Lemma~\ref{lem2.1}, we can easily obtain the following result:

\begin{lemma}\label{lem2.2}
Given any $n$-length word $(i_0,\dotsc,i_{n-1})\in\{1,\dotsc,K\}^n$ for $n\ge1$, it is $\mathbb{S}$-periodically extendable if and only if $(S_{i_0},\dotsc,S_{i_{n-1}})\in\bS^n$
is a non-ignorable closed sample path of $\bxi$.
\end{lemma}

Let
\begin{equation*}
\pi_{\bxi}\colon\Omega\rightarrow\{1,\dots,K\}^{\mathbb{Z}_+}
\end{equation*}
be the natural coding random variable defined by
\begin{equation*}
\omega\mapsto(i_n)_{n=0}^\infty,\quad \textrm{where }(\xi_n(\omega))_{n=0}^\infty=(S_{i_n})_{n=0}^\infty.
\end{equation*}
And let $\mathbb{P}^{\pi_{\bxi}}$ denote the probability distribution of $\pi_{\bxi}$ on $\{1,\dots,K\}^{\mathbb{Z}_+}$.

The following lemma is basic for proving Theorems~\ref{thm1} and \ref{thm2}.

\begin{lemma}\label{lem2.3}
For $\mathbb{P}$-a.e. $\omega\in\Omega$, $\pi_{\bxi}(\omega)$ belongs to $\varSigma_{\mathbb{S}}^+$. In other words, $\mathbb{P}^{\pi_{\bxi}}(\varSigma_{\mathbb{S}}^+)=1$.
\end{lemma}

\begin{proof}
Let $W^n(\mathbb{S})$ be the set of all $n$-length $\mathbb{S}$-admissible words $(i_0,\dotsc,i_{n-1})\in\{1,\dotsc,K\}^n$ for any $n\ge2$. Set
\begin{equation*}
\Omega_{i_0,\dotsc,i_{n-1}}=\{\xi_0=S_{i_0},\dotsc,\xi_{n-1}=S_{i_{n-1}}\}.
\end{equation*}
By Lemma~\ref{lem2.1}, we can get that $\{\Omega_{i_0,\dotsc,i_{n-1}}\colon (i_0,\dotsc,i_{n-1})\in W^n(\mathbb{S})\}$ is a measurable partition of $\Omega$ mod $0$. Let
$\Omega_n=\bigcup_{w\in W^n(\mathbb{S})}\Omega_w$. Then $\Omega_\infty=\bigcap_{n\ge2}\Omega_n$ is of $\mathbb{P}$-measure $1$ and $\pi_{\bxi}(\omega)$ belongs to $\varSigma_{\mathbb{S}}^+$ for each $\omega\in\Omega_\infty$.

This completes the proof of Lemma~\ref{lem2.3}.
\end{proof}

Given any $(i_n)_{n=0}^\infty\in\varSigma_{\mathbb{S}}^+$, there corresponds an infinite sequence of matrices $(S_{i_n})_{n=0}^\infty\in\bS^{\mathbb{Z}_+}$. Now $\bS$ is said to be \textit{uniformly exponentially stable governed by $\mathbb{S}$} if there are constants $C>0$ and $\gamma<1$ such that
\begin{equation*}
\|S_{i_0}\dotsm S_{i_{n-1}}\|\le C\gamma^n\quad \forall (i_n)_{n=0}^\infty\in\varSigma_{\mathbb{S}}^+.
\end{equation*}
It is called \textit{exponentially stable $\mathbb{P}^{\pi_{\bxi}}$-almost surely} if
\begin{equation*}
\lambda_{\bS}((i_n)_{n=0}^\infty):=\lim_{n\to\infty}\frac{1}{n}\log\|S_{i_0}\dotsm S_{i_{n-1}}\|<0\quad \textrm{for }\mathbb{P}^{\pi_{\bxi}}\textrm{-a.e. }(i_n)_{n=0}^\infty\in\varSigma_{\mathbb{S}}^+.
\end{equation*}

Then from Lemma~\ref{lem2.3}, there follows the following two lemmas.

\begin{lemma}\label{lem2.4}
If $\bS$ is uniformly exponentially stable governed by $\mathbb{S}$, then $\bxi$ is uniformly exponentially stable.
\end{lemma}

\begin{lemma}\label{lem2.5}
If $\bS$ is exponentially stable $\mathbb{P}^{\pi_{\bxi}}$-almost surely, then $\bxi$ is exponentially stable $\mathbb{P}$-almost surely.
\end{lemma}

By $W^n(\mathbb{S})$, we mean the set of all $n$-length $\mathbb{S}$-admissible words $(i_0,\dotsc,i_{n-1})\in\{1,\dotsc,K\}^n$ for any $n\ge2$, as before. By $W_{\textit{per}}^n(\mathbb{S})$ is meant the set of all $n$-length $\mathbb{S}$-periodically extendable words $(i_0,\dotsc,i_{n-1})$ in $\{1,\dotsc,K\}^n$, for any $n\ge1$.

Thus from Lemmas~\ref{lem2.1} and \ref{lem2.4}, to prove Theorem~\ref{thm1} it is sufficient to show the following.

\begin{thm}\label{thm2.6}
The following statements are equivalent to each other:
\begin{enumerate}
\item[$(1)$] $\bS$ is uniformly exponentially stable governed by $\mathbb{S}$.
\item[$(2)$] There are constants $\gamma<1$ and $N>0$ such that for each $n>N$,
$\rho(S_{i_0}\dotsm S_{i_{n-1}})\le\gamma$ for all $(i_0,\dotsc,i_{n-1})\in W^n(\mathbb{S})$.
\item[$(3)$] There are constants $\gamma<1$ and $N>0$ such that for each $n>N$,
$\rho(S_{i_0}\dotsm S_{i_{n-1}})\le\gamma$ for all $(i_0,\dotsc,i_{n-1})\in W_{\textit{per}}^n(\mathbb{S})$.
\end{enumerate}
\end{thm}

We note that $(1)\Leftrightarrow(2)$ has already been proved under the additional condition that $\bS$ is product bounded, i.e., there is a constant $\beta>0$ such that $\|A\|\le\beta$ for all $A\in\bS^n$ and any $n\ge1$; see \cite[Theorem~B]{Dai-LAA}.

Similarly, to prove Theorem~\ref{thm2} it is sufficient to show the following.

\begin{thm}\label{thm2.7}
Let the $\{0,1\}$-matrix $\mathbb{S}$ be irreducible. If $\rho(S_{i_0}\dotsm S_{i_{n-1}})<1$ for all $(i_0,\dotsc,i_{n-1})$ in $W_{\textit{per}}^n(\mathbb{S})$ and for each $n\ge1$,
then $\bS$ is exponentially stable $\mathbb{P}^{\pi_{\bxi}}$-almost surely.
\end{thm}

Recall here that the irreducibility of $\mathbb{S}$ means that for any pair $1\le i,j\le K$, the $(i,j)$-entry of the product matrix $\mathbb{S}^N$ is strictly positive for some positive integer $N=N_{i,j}$.

Theorem~\ref{thm2.6} positively answers \cite[Question~3]{Dai-LAA} in the situation of matrix-valued topological Markovian chains, and Theorem~\ref{thm2.7} is an extension of Main Theorem of \cite{DHX} from fullshift to subshift of finite-type.

We shall prove the above two theorems in Section~\ref{sec4} after introducing some necessary tools.

\section[The lift of a matrix-valued topological Markovian chain]{The Kozyakin $\{0,1\}$-matrix lift of a matrix-valued topological Markovian chain}\label{sec3}
This section will be devoted to introducing our main tool---the $\{0,1\}$-matrix lift of a matrix-valued topological Markovian chain---following V.~Kozyakin's idea~\cite{Koz-pre}.

Let $\mathbb{S}=(s_{ij})$ be a $K\times K$ matrix of $0$s and $1$s such that each row of $\mathbb{S}$ contains at least one entry $1$, and let $\bS=\{S_1,\dotsc,S_K\}\subset\mathbb{R}^{d\times d}$. By $\mathbb{S}_i$ we denote the $i^{\textrm{th}}$-row of $\mathbb{S}$, for $1\le i\le K$. Let $\bdelta_i=(\delta_{i1},\dotsc,\delta_{iK})$ be the $i^{\textrm{th}}$-row of the $K\times K$ unit matrix, where $\delta_{ik}$ is the Kronecker symbol. Set $\mathbb{S}^{(i)}=\bdelta_i^T\mathbb{S}_i$, which is a $K\times K$ matrix, for $1\le i\le K$. For example, let
\begin{equation*}
    \mathbb{S}=\left(\begin{matrix}0&1&1\\1&0&1\\1&1&0\end{matrix}\right);\; \textrm{ then }\mathbb{S}^{(1)}=\left(\begin{matrix}0&1&1\\0&0&0\\0&0&0\end{matrix}\right),\;\mathbb{S}^{(2)}=\left(\begin{matrix}0&0&0\\1&0&1\\0&0&0\end{matrix}\right)\textrm{ and }\,\mathbb{S}^{(3)}=\left(\begin{matrix}0&0&0\\0&0&0\\1&1&0\end{matrix}\right).
\end{equation*}
We note here that Kozyakin defined $\mathbb{S}^{(i)}=\mathbb{S}_i^T\bdelta_i$ in a slightly different way~\cite{Koz-pre}; and $^T$ means the transpose operator of matrices.

Given any two matrices $A=(a_{ij})\in\mathbb{R}^{K\times K}$ and $B=(b_{ij})\in\mathbb{R}^{d\times d}$, the Kronecker product $A\otimes B$ is defined as the block matrix
\begin{equation*}
    \left(\begin{matrix}a_{11}B&\dotsm&a_{1K}B\\\vdots&\dotsm&\vdots\\a_{K1}B&\dotsm&a_{KK}B\end{matrix}\right)\in\mathbb{R}^{Kd\times Kd},
\end{equation*}
whose entries are $d\times d$ matrices; see \cite{HJ}. We now define
\begin{equation*}
    \mathbb{S}\otimes\bS=\left\{S^{(1)},\dotsc,S^{(K)}\right\},\quad \textrm{where }S^{(k)}=\mathbb{S}^{(k)}\otimes S_k\textrm{ for }1\le k\le K,
\end{equation*}
which is called the \textit{Kozyakin $\mathbb{S}$-lift} of the system $\bS$. It was first introduced by V.~Kozyakin in~\cite{Koz-pre} to prove the Gel'fand-Berger-Wang formula of a matrix-valued topological Markovian chain; see Theorem~\ref{thm4.1} below.

Recall that for any word $(i_0,\dotsc,i_{n-1})\in\{1,\dotsc,K\}^n$ where $n\ge2$, it is said to be $\mathbb{S}$-admissible if $s_{i_{k}i_{k+1}}=1$ for all $0\le k\le n-2$.
The following three results are very important for proving Theorems~\ref{thm2.6} and \ref{thm2.7}.

\begin{thm}[{Kozyakin~\cite{Koz-pre}}]\label{thm3.1}
If a word $(i_0,\dotsc,i_{n-1})\in\{1,\dotsc,K\}^n$, where $n\ge2$, is not $\mathbb{S}$-admissible, then $S^{(i_0)}\dotsm S^{(i_{n-1})}=0\in\mathbb{R}^{Kd\times Kd}$.
\end{thm}

\begin{thm}[{Kozyakin~\cite{Koz-pre}}]\label{thm3.2}
For any word $(i_0,\dotsc,i_{n-1})\in\{1,\dotsc,K\}^n$ where $n\ge2$, it holds that
\begin{equation*}
    s_{i_ni_0}\rho(S_{i_0}\dotsm S_{i_{n-1}})=\rho\left(S^{(i_0)}\dotsm S^{(i_{n-1})}\right)
\end{equation*}
whenever $(i_0,\dotsc,i_{n-1})$ is $\mathbb{S}$-admissible.
\end{thm}

Note that a word $(i_0,\dotsc,i_{n-1})$ is $\mathbb{S}$-periodically extendable if and only if $(i_0,\dotsc,i_{n-1},i_0)$ is $\mathbb{S}$-admissible. As a result of the above Theorem~\ref{thm3.1}, we can obtain the following useful fact.

\begin{cor}\label{cor3.3}
If a word $(i_0,\dotsc,i_{n-1})\in\{1,\dotsc,K\}^n$, where $n\ge1$, is not $\mathbb{S}$-periodically extendable, then $\rho\left(S^{(i_0)}\dotsm S^{(i_{n-1})}\right)=0$.
\end{cor}

\begin{proof}
Let $(i_0,\dotsc,i_{n-1})$ be not $\mathbb{S}$-periodically extendable. Then $(i_0,\dotsc,i_{n-1}, i_0)$ is not $\mathbb{S}$-admissible. Thus by Theorem~\ref{thm3.1}, we have
$S^{(i_0)}\dotsm S^{(i_{n-1})}S^{(i_0)}=0$. Then
\begin{equation*}
    \rho\left(S^{(i_0)}\dotsm S^{(i_{n-1})}\right)=\sqrt[2]{\rho\left(S^{(i_0)}\dotsm S^{(i_{n-1})}S^{(i_0)}\dotsm S^{(i_{n-1})}\right)}=0.
\end{equation*}
This completes the proof of Corollary~\ref{cor3.3}.
\end{proof}

For the proof of Theorem~\ref{thm3.1}, readers can see \cite[Lemma~1]{Koz-pre}. The statement of Theorem~\ref{thm3.2} is contained in the proof of Kozyakin~\cite[Theorem~1]{Koz-pre}.

\section{Stability of a matrix-valued topological Markovian chain}\label{sec4}
This section will be devoted to proving Theorems~\ref{thm1} and \ref{thm2} stated in Section~\ref{sec1.2} via proving Theorems~\ref{thm2.6} and \ref{thm2.7} stated in Section~\ref{sec2}, using Kozyakin's $\{0,1\}$-matrix lift approach developed for matrix-valued topological Markovian chains in \cite{Koz-pre}.

Let $\bS=\{S_1,\dotsc,S_K\}\subset\mathbb{R}^{d\times d}$ and $\mathbb{S}$ a $\{0,1\}$-matrix of $K\times K$, not necessarily irreducible, such that each row contains at least one $1$ as the transition sign matrix of the Markovian chain $\bxi$ in Section~\ref{sec1}. Let
\begin{equation*}
    \rho(\bS,\mathbb{S})=\limsup_{n\to\infty}\max_{(i_0,\dotsc,i_{n-1})\in W_{\textit{per}}^n(\mathbb{S})}\sqrt[n]{\rho(S_{i_0}\dotsm S_{i_{n-1}})}
\end{equation*}
and
\begin{equation*}
    \hat{\rho}(\bS,\mathbb{S})=\lim_{n\to\infty}\max_{(i_0,\dotsc,i_{n-1})\in W^n(\mathbb{S})}\sqrt[n]{\|S_{i_0}\dotsm S_{i_{n-1}}\|}
\end{equation*}
which are called the generalized and joint spectral radius of $\bS$ governed by $\mathbb{S}$, respectively.

Using the ergodic theory, an analog of the classical Berger-Wang formula~\cite{BW} is the following statement.

\begin{thm}[\cite{Dai-pre}]\label{thm4.1}
$\hat{\rho}(\bS,\mathbb{S})={\rho}(\bS,\mathbb{S})$.
\end{thm}

This implies that for any given $\{0,1\}$-matrix $\mathbb{S}\in\mathbb{R}^{K\times K}$, ${\rho}(\bS,\mathbb{S})$ is continuous with respect to $\bS$ in $\stackrel{K\textrm{-folds}}{\overbrace{\mathbb{R}^{d\times d}\times\dotsm\times\mathbb{R}^{d\times d}}}$; see \cite[Corollary~1.5]{Dai-pre}. Based on the classical Berger-Wang formula~\cite{BW}, a matrix theory proof of this formula is available in \cite{Koz-pre}.

\subsection{Uniform exponential stability}
To prove Theorem~\ref{thm2.6}, we will need the following known sufficient and necessary condition for uniform exponential stability of $\bS$ governed by $\mathbb{S}$.

\begin{lemma}[\cite{Dai-LAA}]\label{lem4.2}
$\bS$ is uniformly exponentially stable governed by $\mathbb{S}$ if and only if ${\rho}(\bS,\mathbb{S})<1$.
\end{lemma}

We will need another known sufficient and necessary condition of stability:
\begin{lemma}[\cite{SWP,Dai-LAA}]\label{lem4.3}
Let $\A=\{A_1,\dotsc,A_K\}\subset\mathbb{R}^{N\times N}$ be arbitrarily given. Then there are constants $C>0$ and $\gamma<1$ such that
\begin{equation*}
\|A_{i_0}\dotsm A_{i_{n-1}}\|\le C\gamma^n\quad\forall n\ge1\textrm{ and }(i_0,\dotsc,i_{n-1})\in\{1,\dotsc,K\}^n,
\end{equation*}
if and only if one can find $\beta<1$ and $M>0$ so that
\begin{equation*}
    \rho(A_{i_0}\dotsm A_{i_{n-1}})\le\beta\quad\forall n\ge M\textrm{ and }(i_0,\dotsc,i_{n-1})\in\{1,\dotsc,K\}^n.
\end{equation*}
\end{lemma}

Now we are ready to prove Theorem~\ref{thm2.6} by using the Kozyakin lift of a matrix-valued topological Markovian chain introduced in Section~\ref{sec3}
 and the results stated above.

\begin{proof}[Proof of Theorem~\ref{thm2.6}]
By the definition of the uniform exponential stability, it is obvious that $(1)\Rightarrow(2)\Rightarrow(3)$. Thus we need only prove the statement that (3)$\Rightarrow$(1).

Let there be given constants $0\le\beta<1$ and $M>2$ such that for each $n>M$, we have that
$\rho(S_{i_0}\dotsm S_{i_{n-1}})\le\beta$ for all $(i_0,\dotsc,i_{n-1})\in W_{\textit{per}}^n(\mathbb{S})$. Then from Corollary~\ref{cor3.3}, Theorems~\ref{thm3.1} and \ref{thm3.2}, it follows that
\begin{equation*}
    \rho\left(S^{(i_0)}\dotsm S^{(i_{n-1})}\right)\le\beta\quad\forall n\ge M\textrm{ and }(i_0,\dotsc,i_{n-1})\in\{1,\dotsc,K\}^n.
\end{equation*}
Therefore by Lemma~\ref{lem4.3}, there are constants $C>0$ and $\gamma<1$ such that
\begin{equation*}
\|S^{(i_0)}\dotsm S^{(i_{n-1})}\|\le C\gamma^n\quad\forall n\ge1\textrm{ and }(i_0,\dotsc,i_{n-1})\in\{1,\dotsc,K\}^n.
\end{equation*}
Next by Corollary~\ref{cor3.3}, Theorems~\ref{thm3.1} and \ref{thm3.2} again, we see that
\begin{equation*}
    \rho(S_{i_0}\dotsm S_{i_{n-1}})\le C\gamma^n\quad\forall n\ge2\textrm{ and }(i_0,\dotsc,i_{n-1})\in W_{\textit{per}}^n(\mathbb{S}).
\end{equation*}
Hence we have ${\rho}(\bS,\mathbb{S})=\hat{\rho}(\bS,\mathbb{S})\le\gamma$ by Theorem~\ref{thm4.1}.

This thus completes the proof of Theorem~\ref{thm2.6} from Lemma~\ref{lem4.2}.
\end{proof}
\subsection{Periodical stability implies almost sure exponential stability}
We now additionally let the $\{0,1\}$-matrix $\mathbb{S}$ be irreducible. We can choose a transition probability matrix $\bP=(p_{ij})$ of $K\times K$ such that its sign matrix is just $\mathbb{S}$. Then by the Perron-Frobenius theorem, one can find a probability vector $\p=(p_1,\dotsc,p_K)$ such that
\begin{equation*}
\p\bP=\p\quad \textrm{and}\quad p_k>0\textrm{ for each }1\le k\le K.
\end{equation*}
We define a canonical Markovian probability measure $\mu_{\p,\bP}$ on the full symbolic sequence space
$$\varSigma_K^+=\{1,\dotsc,K\}^{\mathbb{Z}_+}$$
as follows:
$$
\mu_{\p,\bP}([j_0,\dotsc,j_{m-1}])=p_{i_0}p_{j_0j_1}\dotsm p_{j_{m-2}j_{m-1}}
$$
for all cylinder sets $[j_0,\dotsc,j_{m-1}]=\{(i_n)_{n=0}^\infty\in\varSigma_K^+\,|\,i_0=j_0,\dotsc,i_{m-1}=j_{m-1}\}$. In fact, here $\mu_{\p,\bP}$ is ergodic since $\bP$ is irreducible.

Let $\mathbb{P}^{\pi_{\bxi}}$ be the probability distribution of $\pi_{\bxi}$ on $(\Omega,\mathscr{F},\mathbb{P})$ valued in $\varSigma_K^+$ as in Lemma~\ref{lem2.3}. The following equivalence is useful.

\begin{lemma}\label{lem4.4}
$\mathbb{P}^{\pi_{\bxi}}$ is equivalent to $\mu_{\p,\bP}$; that is, $\mathbb{P}^{\pi_{\bxi}}(B)=0$ if and only if $\mu_{\p,\bP}(B)=0$, for any Borel subset $B\subseteq\varSigma_K^+$.
\end{lemma}

\begin{proof}
We need only to check that $\mathbb{P}^{\pi_{\bxi}}([i_0,\dotsc,i_{n-1}])=0$ if and only if $\mu_{\p,\bP}([i_0,\dotsc,i_{n-1}])=0$, for any $[i_0,\dotsc,i_{n-1}]\subseteq\varSigma_K^+$.

Let $\p^{(0)}=\left(p_1^{(0)},\dotsc,p_K^{(0)}\right)$ be the irreducible initial probability distribution of $\bxi$. Noting that
\begin{equation*}\begin{split}
\mathbb{P}^{\pi_{\bxi}}([i_0,\dotsc,i_{n-1}])&=\mathbb{P}(\{\xi_0=S_{i_0},\dotsc,\xi_{n-1}=S_{i_{n-1}}\})\\
&=p_{i_0}^{(0)}p_{i_0i_1}(0)\dotsm p_{i_{n-2}i_{n-1}}(n-2)
\end{split}\end{equation*}
and
\begin{equation*}
\mu_{\p,\bP}([i_0,\dotsc,i_{n-1}])=p_{i_0}p_{i_0i_1}\dotsm p_{i_{n-2}i_{n-1}},
\end{equation*}
from $\mathrm{sign}(p_k)=\mathrm{sign}(p_k^{(0)})=1$ for all $1\le k\le K$ and
$$\mathrm{sign}(p_{i_0i_1}(0))=\mathrm{sign}(p_{i_0i_1}),\dotsc,\mathrm{sign}(p_{i_{n-2}i_{n-1}}(n-2))=\mathrm{sign}(p_{i_{n-2}i_{n-1}})$$
the statement follows immediately.

This completes the proof of Lemma~\ref{lem4.4}.
\end{proof}

\begin{lemma}[\cite{DHX}]\label{lem4.5}
Let $\A=\{A_1,\dotsc,A_K\}\subset\mathbb{R}^{N\times N}$ be arbitrarily given. If $\rho(A_{i_0}\dotsm A_{i_{n-1}})<1$ for all words $(i_0,\dotsc,i_{n-1})\in\{1,\dotsc,K\}^n$ and $n\ge1$, then
\begin{equation*}
\lim_{n\to\infty}\frac{1}{n}\log\|A_{i_0}\dotsm A_{i_{n-1}}\|<0
\end{equation*}
for $\mu_{\p,\bP}$-a.e. $(i_n)_{n=0}^\infty\in\varSigma_K^+$.
\end{lemma}

We note here that even if $\bP$ is not irreducible, the statement of lemma~\ref{lem4.5} also holds (cf.~\cite[Proposition~2.5]{DHX-IEEE}).

We can now prove Theorem~\ref{thm2.7} by using Kozyakin's lift of a matrix-valued topological Markovian chain and the above two lemmas.

\begin{proof}[Proof of Theorem~\ref{thm2.7}]
According to Lemma~\ref{lem4.4}, it is sufficient to prove that $\bS$ is exponentially stable $\mu_{\p,\bP}$-almost surely.

From Corollary~\ref{cor3.3}, Theorems~\ref{thm3.1} and \ref{thm3.2}, it follows that
\begin{equation*}
    \rho\left(S^{(i_0)}\dotsm S^{(i_{n-1})}\right)<1\quad\forall n\ge 1\textrm{ and }(i_0,\dotsc,i_{n-1})\in\{1,\dotsc,K\}^n.
\end{equation*}
Then by Lemma~\ref{lem4.5}, it follows that $\mathbb{S}\otimes\bS$ is exponentially stable $\mu_{\p,\bP}$-almost surely. Since
$$
S^{(i_0)}\dotsm S^{(i_{n-1})}=s_{i_0i_1}\cdot\dotsm\cdot s_{i_{n-2}i_{n-1}}\left(\bdelta_{i_0}^T\mathbb{S}^{(i_{n-1})}\right)\otimes(S_{i_0}\dotsm S_{i_{n-1}})
$$
and $\bdelta_{i_0}^T\mathbb{S}_{i_{n-1}}\not=0\in\mathbb{R}^{K\times K}$ for the ${i_{n-1}}^{\textrm{th}}$-row of $\mathbb{S}$ contains at least one entry $1$, it holds that
\begin{equation*}\begin{split}
\|S^{(i_0)}\dotsm S^{(i_{n-1})}\|&\ge\|\left(\bdelta_{i_0}^T\mathbb{S}_{i_{n-1}}\right)\otimes(S_{i_0}\dotsm S_{i_{n-1}})\|\\
&\ge\|S_{i_0}\dotsm S_{i_{n-1}}\|.
\end{split}\end{equation*}
Hence $\bS$ is exponentially stable $\mu_{\p,\bP}$-almost surely.

This completes the proof of Theorem~\ref{thm2.7}.
\end{proof}

Therefore, we have proved our Theorems~\ref{thm1} and \ref{thm2} stated in Section~\ref{sec1.2}.

\section{Concluding remarks}\label{sec5}
In this note, we have studied the uniform and a.e. exponential stability of a nonhomogeneous Markovian chain $\bxi=(\xi_n)_{n\ge0}$ defined on a probability space $(\Omega,\mathscr{F},\mathbb{P})$ valued in a finite set $\bS=\{S_1,\dotsc,S_K\}$ of $d\times d$ matrices. Although the transition probability matrices are not constant and hence $\bxi$ is not necessarily to be stationary, yet if they have the same transition sign matrix we have shown the following two statements:
\begin{itemize}
\item $\bxi$ is uniformly exponentially stable if and only if it is completely periodically stable; i.e., there exists a constant $\gamma<1$ such that for any $n\ge1$, $\rho(S_{i_0}\dotsm S_{i_{n-1}})\le\gamma$ for all non-ignorable closed sample paths $(S_{i_0},\dotsc,S_{i_{n-1}})\in\bS^n$ of $\bxi$.

\item Irreducible $\bxi$ is exponentially stable $\mathbb{P}$-almost surely if it is periodically stable; i.e., for any $n\ge1$, $\rho(S_{i_0}\dotsm S_{i_{n-1}})<1$ for all non-ignorable closed sample paths $(S_{i_0},\dotsc,S_{i_{n-1}})\in\bS^n$ of $\bxi$.
\end{itemize}
These statements provide us characterizations of exponential stability of nonstationary matrix-valued Markovian chains. By Theorems~\ref{thm2.6} and \ref{thm2.7}, we see that the stability of $\bxi$ does not depend on
the explicit values of the transition probability matrices $\bP(n)$, but it depends only upon its transition sign matrix $\mathbb{S}$ of $\bxi$.

Finally we conclude this note with the following open problem for our further researching:

\begin{que}
Let $\bxi$ be exponentially stable $\mathbb{P}$-almost surely. Does it holds that for any $n\ge1$ and all non-ignorable \uwave{closed} sample paths $(S_{i_0},\dotsc,S_{i_{n-1}})\in\bS^n$ of $\bxi$, $\rho(S_{i_0}\dotsm S_{i_{n-1}})<1$?
\end{que}

We note here that if we abusedly require more: $\rho(S_{i_0}\dotsm S_{i_{n-1}})<1$ for all non-ignorable sample paths $(S_{i_0},\dotsc,S_{i_{n-1}})\in\bS^n$ of $\bxi$, then the statement is not necessarily to be true as shown by the example considered in Section~\ref{sec1.1}.
\section*{\textbf{Acknowledgments}}%
The authors would like to thank Professor Victor~Kozyakin for many helpful discussion.

This publication was made possible by NPRP grant $\#$[4-1162-1-181] from the Qatar National Research Fund (a member of Qatar Foundation). The statements made herein are solely the responsibility of the authors.

Dai was supported partly by National Natural Science Foundation of China (No. 11271183) and PAPD of Jiangsu Higher Education Institutions. Y.~Huang was supported partly by National Natural Science Foundation of China (No. 11371380).


\bibliographystyle{amsplain}

\begin{thebibliography}{10}
\bibitem{Bar}
   \newblock {N.~Barabanov},
   \newblock {Lyapunov indicators of discrete inclusions I--III},
   \newblock {Autom. Remote Control 49 (1988), 152--157, 283--287, 558--565}.

\bibitem{BW}
    \newblock {M.\,A.~Berger, Y.~Wang},
    \newblock {Bounded semigroups of matrices},
    \newblock {Linear Algebra Appl. {166} (1992)~21--27}.

\bibitem{BTT}
  \newblock {V.\,D.~Blondel, J.~Theys, J.\,N.~Tsitsiklis},
  \newblock {When is a pair of matrices stable?}
  \newblock {in Unsolved Problems in Mathematical Systems and Control Theory, Ed. V.\,D.~Blondel and A.~Megretski, Princeton University Press, Princeton, NJ, 2004}.

\bibitem{Dai-LAA}
   \newblock {X.~Dai},
   \newblock {A Gel'fand-type spectral-radius formula and stability of linear constrained switching systems},
   \newblock {Linear Algebra Appl. {436} (2012) 1099--1113}.

\bibitem{Dai-pre}
   \newblock {X.~Dai},
   \newblock {Robust periodic stability implies uniform exponential stability of Markovian jump linear systems and random linear ordinary differential equations},
   \newblock {J. Franklin Institue, DOI: 10.1016/j.jfranklin.2014.01.010. ArXiv: 1307.4209 [math.DS]}.

\bibitem{DHX}
   \newblock {X.~Dai, Y.~Huang, M.~Xiao},
   \newblock {Periodically switched stability induces exponential stability of discrete-time linear switched systems in the sense of Markovian probabilities},
   \newblock {Automatica 47 (2011) 1512--1519}.

\bibitem{DHX-IEEE}
   \newblock {X.~Dai, Y.~Huang, M.~Xiao},
   \newblock {Pointwise stabilization of discrete-time stationary matrix-valued Markovian processes},
   \newblock {arXiv: 1107.0132v2 [math.PR], to appear in IEEE Trans. Automat. Control}.

\bibitem{Els}
   \newblock {L.~Elsner},
   \newblock {The generalized spectral-radius theorem: an analytic-geometric proof},
   \newblock {Linear Algebra Appl. {220} (1995) 151--159}.

\bibitem{Gur}
  \newblock {L.~Gurvits},
  \newblock {Stability of discrete linear inclusions},
  \newblock {Linear Algebra Appl. {231} (1995) 47--85}.

\bibitem{HJ}
   \newblock {R.\,A.~Horn, C.\,R.~Johnson},
   \newblock {Topics in Matrix Analysis},
   \newblock {Cambridge University Press, Cambridge, 1994}.

\bibitem{Koz-pre}
   \newblock {V.~Kozyakin},
   \newblock {The Berger-Wang formula for the Markovian joint spectral radius},
   \newblock {ArXiv: 1401.2711 [math.PA], to appear in Linear Algebra Appl}.

\bibitem{LM}
   \newblock {D.~Liberzon, A.\,S.~Morse},
   \newblock {Basic problems in stability and design of switched systems},
   \newblock {IEEE Control Syst. Mag. 19 (1999) 59--70}.

\bibitem{PR}
   \newblock {E.\,S.~Pyatnitski\v{i}, L.\,B.~Rapoport},
   \newblock {Periodic motion and tests for absolute stability on nonlinear nonstationary systems},
   \newblock {Autom. Remote Control 52 (1991), 1379--1387}.

\bibitem{SWP}
   \newblock {M.-H.~Shih, J.-W.~Wu, C.-T.~Pang},
   \newblock {Asymptotic stability and generalized Gelfand spectral radius formula},
   \newblock {Linear Algebra Appl. {252} (1997) 61--70}.

\bibitem{SWMWK}
   \newblock {R.~Shorten, F.~Wirth, O.~Mason, K.~Wulff, C.~King},
   \newblock {Stability criteria for switched and hybrid systems},
   \newblock {SIAM Rev. {49} (2007), 545--592}.
\end{thebibliography}

\end{document}